\documentclass[a4paper, 12pt]{article}
\usepackage{fullpage}
\usepackage{amsmath, amsthm, amssymb, mathrsfs, graphicx, subfigure}
\usepackage{ifthen, url}
\usepackage[usenames]{color}

\usepackage[sort&compress]{natbib}
\bibpunct{(}{)}{;}{a}{}{,}

\theoremstyle{plain}
\newtheorem{thm}{Theorem}
\newtheorem{lem}{Lemma}
\newtheorem{prop}{Proposition}
\newtheorem{cor}{Corollary}

\theoremstyle{definition}
\newtheorem{defn}{Definition}

\theoremstyle{remark}
\newtheorem{remark}{Remark}
\newtheorem*{astep}{A-step}
\newtheorem*{pstep}{P-step}
\newtheorem*{cstep}{C-step}

\theoremstyle{remark}
\newtheorem*{toy}{Toy Example}

\newcommand{\A}{\mathscr{A}}
\newcommand{\X}{\mathscr{X}}

\newcommand{\U}{\mathscr{U}}

\newcommand{\J}{\mathcal{J}}

\newcommand{\KK}{\mathbb{K}}

\newcommand{\RR}{\mathbb{R}}
\newcommand{\UU}{\mathbb{U}}
\newcommand{\XX}{\mathbb{X}}

\newcommand{\nm}{{\sf N}}
\newcommand{\unif}{{\sf Unif}}
\newcommand{\tdist}{{\sf t}}
\newcommand{\fdist}{{\sf F}}
\newcommand{\chisq}{{\sf ChiSq}}

\newcommand{\M}{\mathcal{M}}

\newcommand{\prob}{\mathsf{P}}

\renewcommand{\S}{\mathcal{S}}
\renewcommand{\SS}{\mathbb{S}}
\newcommand{\bel}{\mathsf{bel}}
\newcommand{\pl}{\mathsf{pl}}
\newcommand{\mpl}{\mathsf{mpl}}

\newcommand{\T}{\mathcal{T}}
\newcommand{\TT}{\mathbb{T}}

\renewcommand{\phi}{\varphi}

\newcommand{\G}{\mathscr{G}}

\newcommand{\vtheta}{\boldsymbol{\theta}}

\newcommand{\stgeq}{\geq_{\text{st}}}

\newcommand{\steq}{=_{\text{st}}}

\begin{document}


\title{Valid uncertainty quantification about the model in a linear regression setting}
\author{
Ryan Martin,\footnote{Department of Mathematics, Statistics, and Computer Science, University of Illinois at Chicago, email: {\tt rgmartin@uic.edu}} \; Huiping Xu,\footnote{Department of Biostatistics, The Richard M.~Fairbanks School of Public Health and School of Medicine, Indiana University, email: {\tt huipxu@iu.edu}} \; Zuoyi Zhang,\footnote{Regenstrief Institute, Inc., email: {\tt zyizhang@indiana.edu}} \; and Chuanhai Liu\footnote{Department of Statistics, Purdue University, email: {\tt chuanhai@purdue.edu}}
}
\date{\today}

\maketitle

\begin{abstract}
In scientific applications, there often are several competing models that could be fit to the observed data, so quantification of the model uncertainty is of fundamental importance.  In this paper, we develop an inferential model (IM) approach for simultaneously valid probabilistic inference over a collection of assertions of interest without requiring any  prior input.  Our construction guarantees that the approach is optimal in the sense that it is the most efficient among those which are valid.  Connections between the IM's simultaneous validity and post-selection inference are also made.  We apply the general results to obtain valid uncertainty quantification about the set of predictor variables to be included in a linear regression model.  

\medskip

\end{abstract}

\section{Introduction}
\label{S:intro}

\subsection{Background}

Linear regression is one of the most widely used statistical tools in scientific applications.  Standard practice is to consider many predictor variables in hopes that only a few will identify themselves as being useful in explaining variation in the response variable.  As a result, there is substantial uncertainty about the underlying model.  The classical frequentist approach to deal with this problem is to use the data to help select a particular model, and there are a plethora of tools available, such as the Akaike information criterion \citep[AIC,][]{akaike1973}, the Bayesian information criterion \citep[BIC,][]{schwarz1978}, lasso \citep{tibshirani1996, tibshirani2011} and its variants, including adaptive lasso \citep{zou2006} and elastic net \citep{zou.hastie.2005}.  Significance tests, such as those in \citet{buhlmann2013} and \citet{lttt2014}, based on these or other methods have also been considered recently.  However, these frequentist methods provide no quantification of the uncertainty about the true model, i.e., they provide no way to conclude that one model being the true model is more plausible than another model being the true model.   

Bayesian methods, on the other hand, are able to produce a summary of the uncertainty among the candidate models; see, for example, \citet{clydegeorge2004} and \citet{ohara2009}.  Starting from a prior distribution on the set of possible models and a set of conditional priors on the model-specific parameters, a posterior distribution on the model space can be obtained via Markov chain Monte Carlo.  Despite the conceptual simplicity of this approach, there are practical difficulties to overcome, including prior specification; in particular, real prior information is rarely available and improper default priors cannot be used.  Without genuine prior information, the model uncertainty assessments coming from the corresponding posterior distribution are not guaranteed to be inferentially meaningful; see Section~\ref{SS:valid.example}.  Perhaps as a result of this lack of meaningfulness, the recent trend in the Bayesian literature on model selection is to de-emphasize uncertainty quantification, opting instead to report only a selected model.  Therefore, modern Bayesian methods are effectively just frequentist selection procedures and they provide no summaries of model uncertainty. 

Researchers have recently started to look beyond the classical frequentist and Bayesian schools for new solutions to challenging problems.  See, for example, recent work on generalized fiducial inference \citep{hannig.review.2015, lai.hannig.lee.2013}, confidence distributions \citep{xie.singh.2012}, and objective Bayes with default, reference, or data-dependent priors \citep{bergerbernardosun2009, fraser2011, fraser.reid.marras.yi.2010, martin.mess.walker.eb}.  Our focus in this paper is the new perspective from the \emph{inferential model} (IM) framework of \citet{imbasics}.  The key feature of this approach is that, for any assertion/hypothesis concerning the unknown parameter, it produces a probabilistic summary of the evidence in data for/against the truthfulness of that assertion.  In addition to this internal or subjective interpretation, these inferential summaries are valid, i.e., suitably calibrated, facilitating an external or objective interpretation.  The technical feature that distinguishes IMs from other approaches is the use of predictive random sets to produce a valid probabilistic summary of uncertainty about the parameter without a prior.

\subsection{Lack of validity: an illustration}
\label{SS:valid.example}

To further motivate our developments here, it will be helpful to elaborate on the claim above that the Bayesian approach does not, in general, provide an inferentially meaningful assessment of model uncertainty.  We consider a simple illustrative example.  Let $X=(X_1,X_2)$ be independent, with $X_i \sim \nm(\theta_i,1)$, $i=1,2$.  There are four models in this case, one for each zero and non-zero combination for $(\theta_1,\theta_2)$; here we will consider two ``model assertions,'' namely, $\{\theta_1=0, \theta_2=0\}$ and $\{\theta_1=0\}$.  A reasonable Bayesian model for this problem assigns a uniform prior over the four candidate models, and a $\nm(0,v)$ prior for the model-specific parameters, where the to-be-specified variance, $v > 0$, controls the degree of prior uncertainty.  The normal prior on the model-specific parameters is just a special case of the familiar g-prior in regression \citep[e.g.,][]{zellner1986}.  

To quantify the uncertainty in the model given the observed data $X=x$, we report the posterior probabilities for the model assertions mentioned above.  To be consistent with the notation and terminology in the rest of the paper, we will write these posterior probabilities as $\pl_x(A)$, where $A$ is one of the model assertions, and interpret this as the ``plausibility'' that model assertion $A$ is true based on the observed data $x=(x_1,x_2)$.  These posterior probabilities are proportional to sums of marginal likelihoods, i.e.,  
\begin{align*}
\pl_x(\theta_1=0, \theta_2=0) & \propto \nm(x_1 \mid 0, 1) \nm(x_2 \mid 0, 1) \\
\pl_x(\theta_1=0) & \propto \nm(x_1 \mid 0, 1) \nm(x_2 \mid 0, 1 + v) + \nm(x_1 \mid 0, 1) \nm(x_2 \mid 0, 1),
\end{align*}
where the normalizing constant is the sum of all four marginal likelihoods.  These are easy to compute, no sophisticated Monte Carlo methods are needed.  The question is if the corresponding uncertainty quantification is meaningful.  For example, the plausibility assigned to the true model should be large, but how large is large?  

Figure~\ref{fig:simple} shows the distribution functions (gray) of these plausibility values for the two assertions of interest, for different true parameter values, and for three different values of $v$.  It is natural to interpret probabilities on the $\unif(0,1)$ scale---e.g., 0.9 is a ``large'' value because it would be exceeded in only 10\% of cases---so we take this as a reference.  There are several key observations.  First, there is a strong dependence on the (arbitrary) choice of prior variance; in fact, depending on $v$, even if the model assertion is true, as in Panels~(a) and (c), it can happen that the posterior probability can never be ``large.'' Second, there is no meaningful way that these distributions can be related to the $\unif(0,1)$ reference typically used for interpretation.  Third, the way that shape of these distributions depends on the assertion and its truth/falsity is haphazard: the distributions in Panels~(a) and (b) look indistinguishable even though the model assertions are true and false, respectively, and the distributions in Panel~(c), based on a different but true assertion, have similar shapes; in Panel~(d), the assertion is false, and the distributions look entirely different from the others.  The take-away message is that the Bayesian posterior model probabilities need not have a meaningful interpretation in terms of quantifying uncertainty about the underlying model.  The example here is relatively simple, so we can expect that similar issues will occur more generally.  

As a preview, Figure~\ref{fig:simple} also shows the distribution function of the model plausibilities obtained from the optimal IM described in Section~\ref{S:optimal.reg}.  The key observation is that the IM-based plausibility is valid in the sense that the comparisons with the $\unif(0,1)$ reference are consistent with our intuition in each case: when the asserted model is true (resp.~false), the plausibility is stochastically no smaller (resp.~no larger) than $\unif(0,1)$.  This important difference compared to the Bayesian output is easiest to see in Panel~(c).  To our knowledge, only the IM approach proposed here provides valid model uncertainty quantification, so this is an important contribution.  

\begin{figure}[t]
\begin{center}
\subfigure[Truth: $\theta_1=\theta_2=0$]{\scalebox{0.45}{\includegraphics{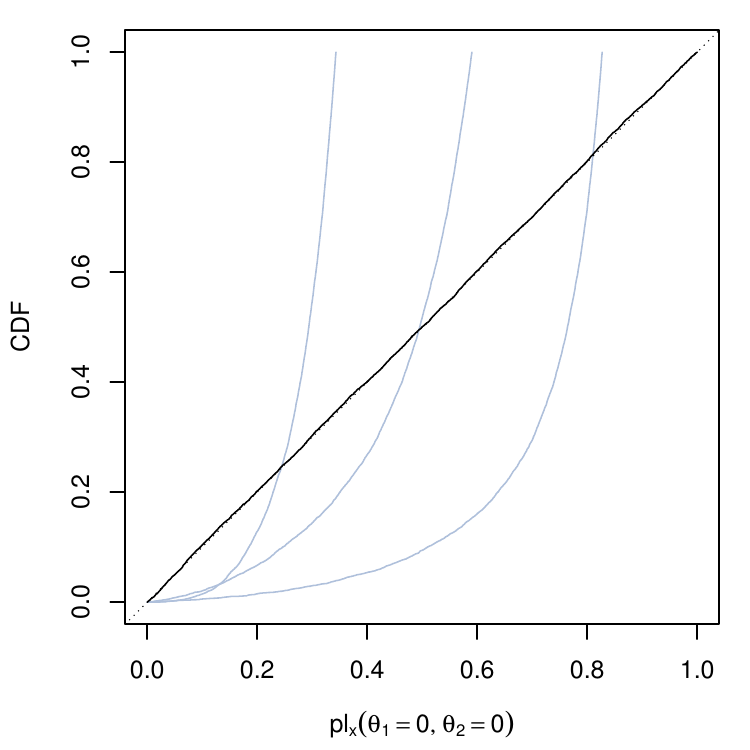}}}
\subfigure[Truth: $\theta_1=0$, $\theta_2=1$]{\scalebox{0.45}{\includegraphics{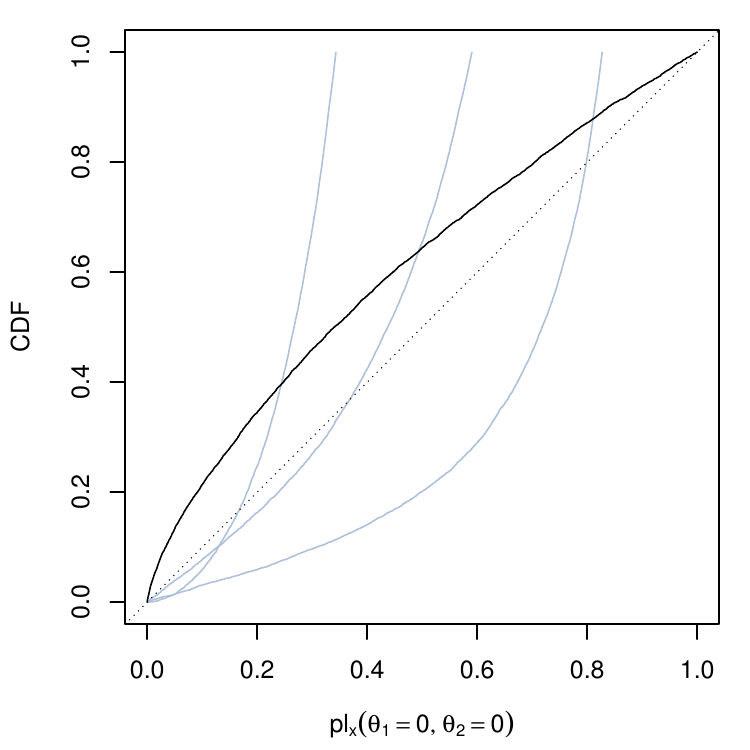}}}
\subfigure[Truth: $\theta_1=0$, $\theta_2=1$]{\scalebox{0.45}{\includegraphics{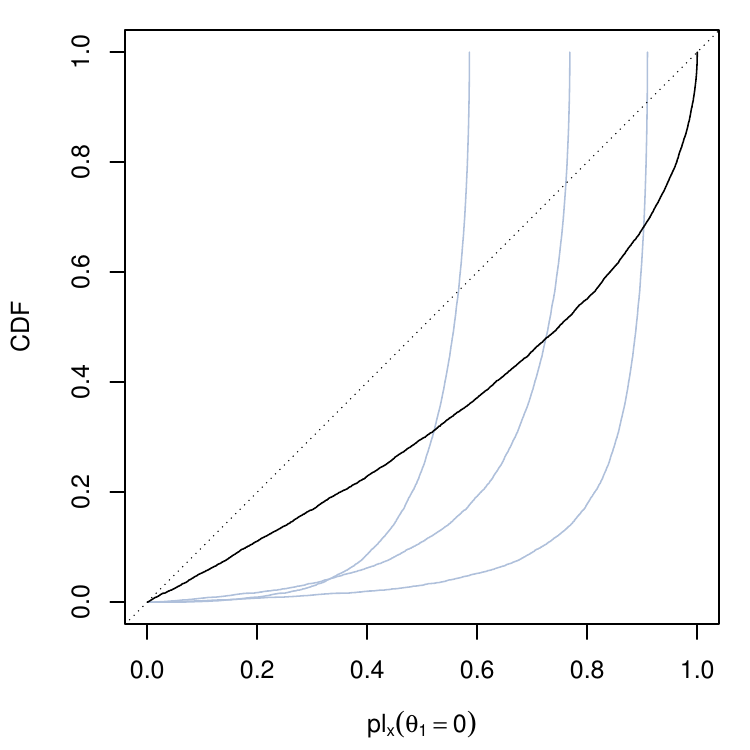}}}
\subfigure[Truth: $\theta_1=2$, $\theta_2=1$]{\scalebox{0.45}{\includegraphics{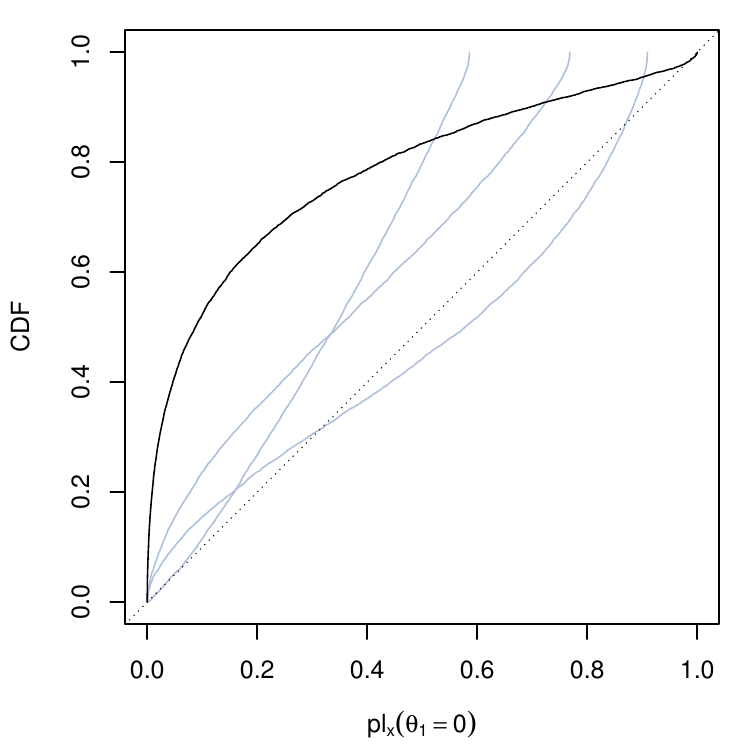}}}
\end{center}
\caption{Plots of the distribution function (CDF) of $\pl_X(\cdot)$ for different model assertions and several configurations of the true $(\theta_1,\theta_2)$.  Gray lines in each panel correspond to the Bayesian plausibility with prior variance $v=1,10,100$; black line correspond to the IM plausibility described in Section~\ref{S:optimal.reg}; and the dotted line is the $\unif(0,1)$ reference.}
\label{fig:simple}
\end{figure}

\subsection{Main contributions}

Early work on IMs for regression is presented in \citet{imreg}.  What is missing there is a rigorous treatment of model uncertainty quantification, and the present paper aims to fill this important gap.  In particular, the main contributions here are two-fold.

First, we present new and general results on the construction of optimal predictive random sets for inference problems that involve multiple simultaneous assertions.  Previous work on optimality \citep[e.g.,][]{imbasics} focused on one assertion at a time but many important examples, such as the regression problem considered here, involve simultaneous consideration of multiple assertions.  Our strategy is to decompose the set of assertions under consideration into their basic building blocks, identify the optimal predictive random sets for these building block assertions, and then combine these individually optimal predictive random sets in a particular way to obtain the simultaneously optimal predictive random set for the collection of assertions.  

Second, we apply these new results to construct an optimal IM for the important problem of model uncertainty quantification in regression.  There we find that a hyper-cube predictive random set is optimal.  The particular construction ensures that the IM's output is provably valid in the sense illustrated in Figure~\ref{fig:simple}.  To our knowledge, this is the first result on valid uncertainty quantification about the model in a regression context.  As part of our development, we make an important connection between the IM's focus on simultaneous validity and modern attempts at post-selection inference, especially, that in \citet{bbbzz2013}.

\subsection{Outline of the paper}

In Section~\ref{S:imreg} we review the basics of the IM approach in the context of the regression problem.  Section~\ref{S:optimal} introduces the concept of efficiency and develops some general theory concerning the shape of the optimal predictive random sets for successively more general collections of assertions, concluding with a result (Theorem~\ref{thm:balance.optimal}) that characterizes the IM for valid and optimal simultaneous inference over a class of so-called complex assertions.  This result is specialized for model uncertainty quantification in regression in Section~\ref{S:optimal.reg}; there we draw connections between notion of simultaneous validity and post-selection inference, and show how the IM's validity property can be used to develop a variable selection procedure with provable control on certain frequentist error rates.   The paper concludes with a short discussion in Section~\ref{S:discuss}; proofs are deferred to the Appendix and some additional technical details are presented in the {Supplementary Material}.

\section{IMs for regression}
\label{S:imreg}



\subsection{Baseline association}

To set notation, consider the following specification of the linear regression model:
\begin{equation}
\label{eq:linreg}
Y = X \beta + \sigma Z, 
\end{equation}
where $Y=(Y_1,\ldots,Y_n)^\top$ is a $n$-vector of response variables, $X$ is an $n \times p$ matrix of predictor variables, $\beta=(\beta_1,\ldots,\beta_p)^\top$ is a $p$-vector of regression coefficients, $\sigma^2$ is the residual variance, and $Z=(Z_1,\ldots,Z_n)^\top$ is a $n$-vector of standard Gaussian errors, i.e., $Z \sim \nm_n(0, I_n)$, a $n$-dimensional Gaussian distribution with mean zero and identity covariance matrix.  We assume that $p < n$ and that $X$ is fixed of full rank.   Without loss of generality, we will assume that $Y$ and the columns of $X$ have been centered, which effectively marginalizes out the nuisance intercept parameter.      

We refer to \eqref{eq:linreg} as a baseline association between observable data $Y$ (and $X$), unknown parameters $(\beta, \sigma)$, and unobservable auxiliary variables $Z$.  The importance of such an association can be seen as follows: if $Z$ could be observed, then we could exactly solve \eqref{eq:linreg} for $(\beta,\sigma)$, leading to the ``best possible inference.''  Since $Z$ cannot be observed, a natural idea is to accurately predict the unobservable $Z$; see Section~\ref{SS:three.step}.  It is clear that accurately predicting a high-dimensional quantity is more difficult than predicting a lower-dimensional quantity, so it is advantageous to reduce the dimension of the auxiliary variable as much as possible before carrying out the prediction step.

\subsection{Auxiliary variable dimension reduction}
\label{SS:dim.red}

\citet{imcond, immarg} discuss two distinct approaches for reducing the dimension of the auxiliary variable.  The first is an approach based on conditioning, which is particularly useful in cases where the auxiliary variable is of higher dimension than the parameter.  For example, in our regression problem, $Z$ is $n$-dimensional while $(\beta,\sigma)$ is $(p+1)$-dimensional, and $p+1 \leq n$.  Since the regression problem admits a $(p+1)$-dimensional minimal sufficient statistic, the IM dimension reduction is straightforward.  Let $\hat\beta$ be the least-squares estimator of $\beta$, and $\hat\sigma$ the corresponding estimator of $\sigma$ based on the residual sum of squares.  Then $(\hat\beta, \hat\sigma)$ is a minimal sufficient statistic, and we can identify a lower-dimensional (conditional) association:
\begin{equation}
\label{eq:conditional.association}
\hat\beta = \beta + \sigma V_1 \quad \text{and} \quad \hat\sigma = \sigma V_2 
\end{equation}
where $V=(V_1,V_2)$ satisfies $V_1 \sim \nm_p(0,M)$, $(n-p-1)V_2^2 \sim \chisq(n-p-1)$, independent, and $M=(X^\top X)^{-1}$.  The key point is that we have replaced the $n$-dimensional auxiliary variable $Z$ with a $(p+1)$-dimensional auxiliary variable $V$.    

Here, as is typical in regression applications, $\sigma^2$ is a nuisance parameter.  For such cases, it is possible to further reduce the dimension of the auxiliary variable via marginalization.  Note that the association \eqref{eq:conditional.association} can be rewritten as 
\[ \hat\beta = \beta + \hat\sigma (V_1 / V_2) \quad \text{and} \quad \hat\sigma = \sigma V_2. \]
The general theory in \citet{immarg} says that the second equation above---the one that is free of $\beta$---can be ignored.  This leaves a marginal association involving a $p$-dimensional auxiliary variable $W=V_1/V_2$.  That is,  
\begin{equation}
\label{eq:marginal.association.1}
\hat\beta = \beta + \hat\sigma W, \qquad W \sim \tdist_p(0,M;n-p-1),
\end{equation}
where the auxiliary variable distribution is a $p$-dimensional Student-t, with $n-p-1$ degrees of freedom, centered at the origin, with scale matrix $M=(X^\top X)^{-1}$.  Again, the important point is that the $(p+1)$-dimensional auxiliary variable in \eqref{eq:conditional.association} has been replaced by a $p$-dimensional auxiliary variable $U$.  No further dimension reduction is possible.   

It will be convenient to rewrite association \eqref{eq:marginal.association.1} once more.  Let $D$ be a diagonal $p \times p$ matrix with the same diagonal as $M$.  Then consider the association 
\begin{equation}
\label{eq:marginal.association}
\hat\theta = \theta + \hat\sigma U, \quad U \sim \prob_U = \tdist_p(0, L, n-p-1),
\end{equation}
where $\theta = D^{-1/2} \beta$, $\hat\theta = D^{-1/2} \hat\beta$, $U = D^{-1/2} W$, and $L = D^{-1/2} M D^{-1/2}$.  Note, in particular, that $\theta_j=0$ if and only if $\beta_j=0$, so the question about which variables are included in the model has not been changed; also, $L$ has all ones on the diagonal.

\subsection{Predictive random sets}
\label{SS:prs}

The interpretation of the association \eqref{eq:marginal.association} from the standpoint of inference is that there is a particular value, say $u^\star$, for which that equation holds for the true $\theta$.  How we describe our uncertainty about this value $u^\star$ determines our uncertainty assessments about $\theta$, given the observed data.  What distinguishes the IM approach from Fisher's fiducial approach and its variants is the use of a random set, $\S$, to summarize uncertainty about $u^\star$.  We refer to $\S$ as a \emph{predictive random set}.  The key to the IM approach's success is an appropriate choice of the predictive random set.  

There is a rich mathematical theory of random sets, summarized nicely in \citet{molchanov2005}.  Here, for conceptual understanding, it will suffice to consider a discrete setting.  Let $\SS$ be a finite collection of subsets of a space $\UU$, in our case, $\RR^p$.  This collection of sets will serve as the support for the random set $\S$; the individual sets in the collection $\SS$ are called \emph{focal elements}.  Now, define a random set $\S$, supported on $\SS$, simply by assigning probabilities to each of the focal elements; this characterizes the distribution, $\prob_\S$, of $\S$.  Of course, defining a random set in a non-discrete case requires more care and, in particular, a notion of measurability of set-valued functions.  The {Supplementary Material} provides some relevant technical background and explains how the random sets appearing in our IM development satisfy the required conditions automatically.  

One is free to describe their uncertainty about the unobserved value $u^\star$ with any kind of predictive random set.  However, the corresponding uncertainty assessments/inference about $\theta$ would be meaningful only if the two distributions in question, namely, $\prob_U$ and $\prob_\S$, have a suitable link.  The following definition provides such a link.    

\begin{defn}
\label{def:admissible}
A predictive random set $\S$, with distribution $\prob_\S$, is called \emph{admissible} if the following two conditions hold:
\begin{itemize}
\item the support $\SS$, which is assumed to contain $\varnothing$ and $\RR^p$, is \emph{nested}, that is, for any two focal elements $S$ and $S'$ in $\SS$, either $S \subseteq S'$ or $S' \subseteq S$, and 
\vspace{-2mm}
\item the distribution $\prob_\S$ of $\S$ satisfies 
\begin{equation}
\label{eq:natural}
\prob_\S\{\S \subset K\} = \sup_{S \in \SS: S \subset K} \prob_U(S), \quad K \subseteq \UU. 
\end{equation}
\end{itemize}
\end{defn}

\citet{imbasics} showed that, by choosing a predictive random set $\S$ that is admissible in the sense above, the corresponding IM output is valid, i.e., suitably calibrated for scientific inference; see Definition~\ref{def:valid} below.  Existence of a distribution $\prob_\S$ that satisfies \eqref{eq:natural} is discussed in the {Supplementary Material}.

\subsection{IM construction}
\label{SS:three.step}

First, the model provides a link between the observable data, the unknown parameters, and the unobservable auxiliary variables.  Next, uncertainty about the unobserved value of the auxiliary variable, given the observed data, is encoded in a predictive random set.  Finally, this uncertainty about the auxiliary variable gets passed through the association, yielding a corresponding uncertainty assessment about the parameter.  For the dimension-reduced model in \eqref{eq:marginal.association}, the formal three-step IM construction is as follows.  

\begin{astep}
Associate data $Y=y$ with unknown parameters $\theta$ and auxiliary variable $U$ as in \eqref{eq:marginal.association.1}.  This defines a set---in this case, a singleton set---of parameter values indexed by particular values of $U$, namely, 
\begin{equation}
\label{eq:marginal.focal}
\Theta_y(u) = \{\theta: \hat\theta = \theta + \hat\sigma u\}, \quad u \in \RR^p.
\end{equation}
\end{astep}

\begin{pstep}
Predict the unobservable auxiliary variable $U$ with an admissible predictive random set $\S$ with distribution $\prob_\S$.  
\end{pstep}

\begin{cstep}
Combine the results of the A- and P-steps to get 
\[ \Theta_y(\S) = \bigcup_{u \in \S} \Theta_y(u). \]
\end{cstep}

Since $\S$ is a random set, so is $\Theta_y(\S)$, and probabilities associated with this random set are used to summarize uncertainty about the unknown parameter.  If $\S$ is such that $\Theta_y(\S) \neq \varnothing$ with $\prob_\S$-probability~1 for all $y$---this is the only case we need to consider here, but see \citet{leafliu2012}---then, for a given assertion/hypothesis $A$ about the unknown parameter $\theta$, we calculate the belief function
\begin{equation}
\label{eq:belief}
\bel_y(A;\S) = \prob_\S\{\Theta_y(\S) \subset A\}, 
\end{equation}
the $\prob_\S$-probability that $\Theta_y(\S)$ completely agrees with the assertion $A$.  The other important quantity is the plausibility function 
\[ \pl_y(A;\S) = 1 - \bel_y(A^c; \S) = \prob_\S\{\Theta_y(\S) \not\subset A^c\}, \]
the $\prob_\S$-probability that $\Theta_y(\S)$ at least partially agrees with $A$.  For given $Y=y$, the output of the IM construction is the pair of functions $(\bel, \pl)_y$.

\subsection{IM validity}
\label{SS:valid}

The output of the IM construction is a belief function, but in what sense is this belief function meaningful for inference?  One important property the belief function should have is that it does not tend to assign high beliefs to assertions about $\theta$ which are false.  The validity property makes this precise.      

\begin{defn}
\label{def:valid}
An IM is called \emph{valid} if, for any assertion $A$ about $\theta$, the corresponding belief function satisfies 
\begin{equation}
\label{eq:bel.valid}
\sup_{(\theta,\sigma): \theta \not\in A} \prob_{Y|(\theta,\sigma)}\{\bel_Y(A;\S) \geq 1-\alpha\} \leq \alpha, \quad \forall \; \alpha \in (0,1). 
\end{equation}
In other words, the IM is valid if, for any assertion $A$, $\bel_Y(A;\S)$ is stochastically no larger than $\unif(0,1)$ as a function of $Y$ when $A$ is false.  
\end{defn}

\citet{imbasics} showed that admissibility of $\S$ is sufficient for validity of the corresponding IM.  The only downside is that there are many predictive random sets that give valid IMs.  In that case, it is natural to look for a ``best'' one; see Definition~\ref{def:optimal}.  

\section{General optimality results}
\label{S:optimal}

\subsection{Setup}

Consider a general setup with association $X=a(\theta,U)$, where $X \in \XX$ is the observable data, $\theta \in \Theta$ is the unknown parameter of interest, and $U$ is an unobservable auxiliary variable, with distribution $\prob_U$ defined on a space $\UU$; precise measure-theoretical statement of our assumptions is given in the {Supplementary Material}.  The corresponding sampling model is denoted by $\prob_{X|\theta}$.  We assume that the predictive random set is admissible, so that the IM's belief function $\bel_x$ is valid in the sense of Definition~\ref{def:valid}.  Then, for a given assertion $A$ about $\theta$, an optimal predictive random set $\S$, if it exists, makes $\bel_X(A;\S)$ stochastically as large as possible as a function of $X \sim \prob_{X|\theta}$, for all $\theta \in A$ \citep{imbasics}.  The following definition makes this formal.  

\begin{defn}
\label{def:optimal}
For a given association and assertion $A$, if $\S$ and $\S'$ are two valid predictive random sets, then we say that $\S$ \emph{is at least as efficient as} $\S'$ (with respect to $A$) if 
\begin{equation}
\label{eq:dominate}
\text{$\bel_X(A;\S) \stgeq \bel_X(A;\S')$, as a function of $X \sim \prob_{X|\theta}$, $\theta \in A$}. 
\end{equation}
A predictive random set $\S$ is called optimal (with respect to $A$) if \eqref{eq:dominate} holds for all valid predictive random sets $\S'$.  
\end{defn}


Given $\Theta_x(u) = \{\theta: x=a(\theta,u)\}$ from the A-step, define the \emph{a-event} as 
\begin{equation}
\label{eq:aevent}
\UU_A(x) = \{u \in \UU: \Theta_x(u) \subseteq A\}.
\end{equation}
This is just the set of all $u$ that support the truthfulness of the assertion ``$\theta \in A$'' for a given $x$.  We require that $\UU_A(x)$ be appropriately measurable, and the precise conditions are given in the {Supplementary Material}.   The a-events will be crucial in the construction of the focal elements of the optimal predictive random set.  

As discussed in Section~\ref{S:intro}, towards developing an optimal predictive random set for a collection of assertions, simultaneously, we consider breaking these assertions down into their basic building blocks.  Starting from the most basic building block, a simple assertion, we characterize the optimal predictive random sets and show how these can be combined to create optimal predictive random sets for more general kinds of assertions.  Since the distribution of an admissible predictive random set is determined by \eqref{eq:natural}, our focus throughout is on the construction of the focal elements.

\subsection{Simple assertions}
\label{SS:simple}

An assertion $A$ is \emph{simple} if the collection of a-events $\UU_A \equiv \{\UU_A(x): x \in \XX\}$ is nested, i.e., for any pair $x,x' \in \XX$, we have either $\UU_A(x) \subseteq \UU_A(x')$ or $\UU_A(x) \supseteq \UU_A(x')$; otherwise, the assertion is called \emph{complex}.  The following results shows that, if $A$ is simple, then the optimal predictive random set (relative to $A$), as in Definition~\ref{def:optimal}, readily obtains by taking the support $\SS=\UU_A$ and the distribution $\prob_\S$ to satisfy \eqref{eq:natural}.   

\begin{thm}
\label{thm:optimal.simple}
For a simple assertion $A$, the optimal predictive random set $\S$ with respect to $A$ is supported on the nested collection $\UU_A$ with distribution satisfying \eqref{eq:natural}.  
\end{thm}

\begin{proof}
See Appendix~\ref{proof:optimal.simple}. 
\end{proof}

As an example, consider $X \sim \nm(\theta,1)$, with association $X=\theta + U$, $U \sim \nm(0,1)$.  The assertion $A=\{\theta > 0\}$ is simple.  To see this, write the a-event:
\[ \UU_A(x) = \{u: \text{$x=\theta+u$ for some $\theta > 0$}\} = (-\infty, x). \]
Then, clearly, if $x < x'$, then $\UU_A(x) \subset \UU_A(x')$; so, the a-events are nested.  The belief function based on the optimal predictive random set $\S$ in Theorem~\ref{thm:optimal.simple} is 
\[ \bel_x(A;\S) = \prob_\S\{\S \subset \UU_x(A)\} = \prob_U\{\UU_A(x)\} = \Phi(x), \]
where $\Phi$ is the standard normal distribution function.  Given the optimality result, it is not a coincidence that $\pl_x(A^c;\S) = 1-\Phi(x)$ is the p-value for the uniformly most powerful test of $H_0: \theta \leq 0$ versus $H_1: \theta > 0$; see \citet{impval}.  

Unfortunately, simple assertions are insufficient for practical applications.  For example, in the normal mean example above, we might be interested in the assertion $A=\{\theta \neq 0\}$ or, more generally, we might be interested in several assertions simultaneously, each of which might be simple or not.  The regression application features all of these cases.  (More generally, even if there is only one assertion $A$, IM efficiency considerations require that one think about $A$ and $A^c$ simultaneously, so an understanding of how to handle multiple assertions is fundamental.)  The next two sections discuss how to extend the basic optimality results to these important more general cases.

\subsection{Complex assertions}
\label{SS:complex}

To motivate our developments, reconsider the normal mean example above.  This time, suppose we are interested in the assertion $A=\{\theta \neq 0\}$.  This is a union of two disjoint simple assertions, namely, $A_1=\{\theta < 0\}$ and $A_2=\{\theta > 0\}$.  The optimal predictive random set with respect to $A_1$ is efficient for $A_1$ but very inefficient for $A_2$; likewise, for the optimal predictive random set with respect to $A_2$.  Since we are interested in $A_1 \cup A_2$ we require a predictive random set that is as efficient as possible for both $A_1$ and $A_2$.  For the general case of a complex assertion written as a union of two simple assertions, an intuitively reasonable strategy is to consider a predictive random set whose focal elements are intersections of the focal elements of the optimal predictive random sets with respect to the two individual simple assertions.  The following result says that the optimal predictive random set for the complex assertion must be of this form.  (The results below extend to more than two simple component assertions, but two-component assertions are general enough for our purposes here.)  

\begin{thm}
\label{thm:intersection}
Let $A=A_1 \cup A_2$ be a complex assertion, where $A_1$ and $A_2$ are simple assertions.  Let $\S_1$ and $\S_2$ be the optimal predictive random sets for $A_1$ and $A_2$, respectively, as in Theorem~\ref{thm:optimal.simple}.  For any predictive random set $\T$, there exists an $\S$, whose focal elements are intersections of the focal elements of $\S_1$ and $\S_2$, such that $\S$ is at least as efficient as $\T$ with respect to $A$.
\end{thm} 

\begin{proof}
See Appendix~\ref{proof:intersection}.
\end{proof}

This result simplifies the search for optimal predictive random sets with respect to a complex assertion.  It does not completely resolve the problem, however, since there are many choices of predictive random sets with intersection focal elements.  In the normal mean problem, for example, the focal elements for the optimal predictive random sets with respect to $A_1$ and $A_2$ are one-sided intervals.  Therefore, the focal elements of the optimal predictive random set for $A=A_1 \cup A_2$ must be nested intervals, but should the intervals be symmetric or asymmetric?   See Section~\ref{SS:towards}.

\subsection{Multiple complex assertions}
\label{SS:multiple}

Suppose there are multiple assertions, $\{A_j: j \in J\}$, to be considered simultaneously, where $A_j = A_{j1} \cup A_{j2}$, $j \in J$, decomposes as a union of two disjoint simple assertions.  Each simple component has an optimal predictive random set according to Theorem~\ref{thm:optimal.simple}, and the corresponding optimal predictive random set, $\S_j$, for the union $A_j$ has intersection focal elements according to Theorem~\ref{thm:intersection}.  As before, intuition suggests that the optimal predictive random set for $\{A_j: j \in J\}$ would be supported on intersections of the focal elements for the individually optimal $\S_j$, $j \in J$.  To justify this intuition, we need a way to measure the efficiency of a predictive random set in this multiple-assertion context.  

\begin{defn}
\label{eq:generate}
An assertion $A$ is said to be \emph{generated by} $\{A_j: j \in J\}$ if $A$ can be written as a union of some or all of the $A_j$s.   Then a predictive random set $\S$ is at least as efficient as $\S'$ with respect to $\{A_j: j \in J\}$ if \eqref{eq:dominate} holds for all $A$ generated by $\{A_j: j \in J\}$.  
\end{defn}

The following result, a generalization of Theorem~\ref{thm:intersection}, shows that a restriction to predictive random sets supported on intersection focal elements is without loss of efficiency.   

\begin{thm}
\label{thm:new.intersection}
Let $\{A_j: j \in J\}$ be a collection of assertions, where $A_j = A_{j1} \cup A_{j2}$ partitions as a union of disjoint simple assertions.  Let $\S_j$ be the optimal predictive random set for $A_j$, $j \in J$.  Then, for any predictive random set $\T$, there exists an $\S$, whose focal elements are intersections of the focal elements of the $\S_j$s, such that $\S$ is at least as efficient as $\T$ with respect to $\{A_j: j \in J\}$.
\end{thm}

\begin{proof}
See Appendix~\ref{proof:new.intersection}.  
\end{proof}

As in the previous section, this result simplifies the search for an optimal predictive random set but does not completely resolve the problem, since it does not determine the particular shape of the focal elements.  Therefore, we need to push further.

\subsection{Balance and optimality}
\label{SS:towards}

To resolve the ambiguity about the shape of the optimal focal elements, here we will introduce some special structure to the problem.  Suppose that there exists transformations of $X$ that do not fundamentally change the inference problem.  As a simple example, if $X \sim \nm(\theta,1)$ and the association of interest is $A=\{\theta \neq 0\}$, then it is clear that changing the sign of $X$ should not have an impact on how much support there is for $A$.  In other words, this normal mean problem is invariant to sign changes.  More generally, let $\G$ be a group of bijections $g$ from $\XX$ to itself; as is customary, we shall write $gx$ for the image of $x$ under $g$, rather than $g(x)$.  Practically, the transformations in $\G$ represent symmetries of the inference problem, i.e., nothing fundamental about the problem changes if $gX$ is observed instead of $X$.  The key technical assumption here is that each $g$ commutes with the association mapping $a$, i.e., 
\begin{equation}
\label{eq:commutes}
g\,a(\theta, u) = a(g\theta, gu), \quad \forall \; (\theta,u), \quad \forall \; g \in \G. 
\end{equation}
Here we are implicitly assuming that $\G$ also acts upon the parameter space $\Theta$ and the auxiliary variable space $\UU$.  This can be relaxed by introducing groups acting on $\Theta$ and $\UU$, respectively, different from (but homomorphic to) $\G$, but we will not need this extra notational complexity for our application here.  The above condition is different from that which defines the usual group transformation models in that, in the right-hand side, $g$ also acts on $u$.  In fact, our variable selection application will not fit the usual group transformation structure, unless the design is orthogonal.  

We shall also require that the assertions respect these symmetries.  Let $A$ be an assertion generated the collection $\{A_j: j \in J\}$ of complex assertions.  Consider the subgroup of $\G$ to which $A$ is invariant, and write $\G_A = \{g \in \G: gA=A\}$.  The intuition is that the inference problem for $\theta$, at least as it concerns the assertion $A$, is unchanged if the problem is transformed by $g \in \G_A$.  

Before proceeding, it may help to see a simple example.  Consider, again, the normal mean problem, with $X=\theta + U$, $U \sim \nm(0,1)$, and assertion $A=\{\theta \neq 0\}$.  Then changing the sign of $X$ will not affect the problem.  So, we can take $\G$ to consist of the identity mapping and $x \mapsto -x$, and clearly \eqref{eq:commutes} holds; also, $\G_A = \G$.  

Moving on, recall the a-event $\UU_A(x) = \{\text{$u$: $x=a(\theta,u)$ for some $\theta \in A$}\}$.  In this case, an equivariance property follows immediately from the commutativity property \eqref{eq:commutes}:
\[ \UU_A(gx) = g \UU_A(x), \quad \forall \; g \in \G_A. \]
That is, given $A$, transforming $x \to gx$, for $g \in \G_A$, and then solving for $u$ is equivalent to solving for $u$ with the given $x$ and then transforming $u \to gu$.  

We now have the necessary structure to help specify an optimal predictive random set.  It suffices to focus on predictive random sets which are admissible.  So, let $\S$ be admissible and, for simplicity, suppose we can write the collection of closed nested focal elements as $\SS=\{S_r: r \in [0,1]\}$, where $r$ corresponds to the set's $\prob_U$-probability; in particular, $\prob_U(S_r) = 1-r$.  Then we have the following useful representation.  

\begin{lem}
\label{lem:belief}
$\prob_{X|\theta}\{\bel_X(A;\S) > 1-r\} = \prob_{X|\theta}\{\UU_A(X) \supset S_r\}$.  
\end{lem}

\begin{proof}
See Appendix~\ref{proof:belief}.
\end{proof}

From Lemma~\ref{lem:belief} and the equivariance property above, if $g \in \G_A$, then 
\[ \prob_{X|\theta}\{\bel_{gX}(A;\S) > 1-r\} \equiv \prob_{X|\theta}\{\UU_A(gX) \supset S_r\} = \prob_{X|\theta}\{\UU_A(X) \supset g^{-1} S_r\}. \]
Since the understanding is that the inference problem, at least as it concerns the assertion $A$, is unchanged by transformations $X \to gX$ for $g \in \G_A$, it is reasonable to require that the distribution of $\bel_X(A;\S)$ be invariant to $\G_A$.  The previous display reveals that the way to achieve belief function invariance is to require the focal elements of the predictive random set to be invariant to $\G$.  This leads to the following notion of \emph{balance}.  

\begin{defn}
\label{def:prs.balance}
The predictive random set $\S$ is said to be \emph{balanced with respect to $A$} if each focal element $S \in \SS$ satisfies $gS=S$ for all $g \in \G_A$.  Moreover, $\S$ is said to be \emph{balanced} if the aforementioned invariance holds for all $g \in \G$.  
\end{defn}

Balance itself is a reasonable property, given that the transformations are, by definition, irrelevant to the inference problem.  It is also interesting and practically beneficial that balance can be checked without doing any probability calculations; however, the focal elements $\SS$ and the transformations $\G$ depend on the model.   

Beyond the intuitive appeal of balance, we claim that balance leads to a particular form of optimality.  Recall that, for IM efficiency, the general goal is to choose the predictive random set to make the belief function stochastically large when the assertion is true.  With this in mind, we propose the following notion of \emph{maximin optimality}.  

\begin{defn}
\label{def:prs.optimal}
A predictive random set $\S^\star$, with focal elements $\{S_r^\star: r \in [0,1]\}$, is \emph{maximin optimal} if it maximizes
\[ \min_{g \in \G_A} \prob_{X|\theta}\{\bel_{gX}(A;\S) > 1-r\} \equiv \min_{g \in \G_A} \prob_{X|\theta}\{\UU_A(X) \supset g^{-1} S_r\} \]
over all admissible $\S$ with focal elements $\{S_r: r \in [0,1]\}$, uniformly for all $A$ generated by $\{A_j: j \in J\}$, for all $\theta \in A$, and for all $r \in [0,1]$.  
\end{defn}

\begin{thm}
\label{thm:balance.optimal}
If a predictive random set $\S$ is balanced in the sense of Definition~\ref{def:prs.balance}, then it is maximin optimal in the sense of Definition~\ref{def:prs.optimal}.  
\end{thm}

The main idea in the proof, presented in Appendix~\ref{proof:balance.optimal}, is the notion of the ``core'' $S^\circ$ of a given focal element $S$.  In particular, the core is defined as $S^\circ = \bigcap_{g \in \G_A} gS$, and the proof relies on the fact that it is both balanced and contained in each $gS$.

\section{Model uncertainty in regression}
\label{S:optimal.reg}

\subsection{Optimal IM for assertions about the model}
\label{SS:vs.assertions}

Recall the dimension-reduced association $\hat\theta = \theta + \hat\sigma U$ in \eqref{eq:marginal.association}, where $U$ is a $p$-vector distributed as $\prob_U=\tdist_p(0, L, n-p-1)$, and $L$ is a matrix with ones on the diagonal.  The intercept has been marginalized out, so each of $\theta_1,\ldots,\theta_p$ can be considered zero or non-zero; the goal is to summarize the uncertainty about these possible models.  

Consider the collection of complex assertions $A_j=\{\theta: \theta_j \neq 0\}$, $j=1,\ldots,p$.  Consider first a particular $A_j$.  This can be written as $A_j = A_{j1} \cup A_{j2}$, where $A_{j1} = \{\theta: \theta_j < 0\}$ and $A_{j2} = \{\theta: \theta_j > 0\}$.  We claim that these sub-assertions are both simple in the sense of Section~\ref{SS:simple}.  Take $A_{j1}$, for example.  Then the corresponding a-event is 
\[ \UU_{A_{j1}}(y) = \{u: \hat\theta - \hat\sigma u \in A_{j1}\} = \{u: u_j > T_j\}, \]
where $T_j = \hat\theta_j/\hat\sigma$.  This a-event is nested because it shrinks and expands monotonically as a function of $T_j$.  The same is true for $A_{j2}$ and for all the other $j=1,\ldots,p$.  Therefore, by Theorem~\ref{thm:optimal.simple}, the optimal predictive random sets for the individual sub-assertions $A_{j1}$ and $A_{j2}$ are each supported on collections of half hyper-planes.  Next, it follows from Theorem~\ref{thm:intersection} that the optimal predictive random set $\S_j$ for the complex assertion $A_j$ is supported on intersections of half hyper-planes, i.e., cylinders $\{u: a_j \leq u_j \leq b_j\}$.  If we are considering $\{A_j: j=1,\ldots,p\}$ simultaneously, then it follows from Theorem~\ref{thm:new.intersection} that the optimal predictive random set is supported on boxes---intersections of $p$ marginal cylinders---in $\RR^p$.  The remaining question is what shape should the boxes be.  

Towards optimality, we need to consider what transformations leave the model uncertainty problem invariant in the sense of Section~\ref{SS:towards}.  As in the simple normal mean example discussed previously, sign changes to coordinates of $\hat\beta$ are irrelevant.  In addition, the labeling of the variables $j=1,\ldots,p$ is arbitrary, so permutations of the variable labels are also irrelevant.  This suggests we consider the group $\G$ of \emph{signed permutations}; that is, each $g \in \G$ acts on a $p$-vector $x$ by matrix multiplication (on the left), where the matrix factors as a product of a diagonal matrix with $\pm 1$ on the diagonal and a permutation matrix.  It is clear that the association commutes with $\G$ in the sense of \eqref{eq:commutes}.  
With the group $\G$ specified, it is also clear what shape of boxes the optimal predictive random set should be supported on.  According to Definition~\ref{def:prs.balance}, a balanced predictive random set should have focal elements---in this case, shaped like boxes in $\RR^p$---that are invariant to the transformations in $\G$.  The only such boxes are hyper-cubes centered at the origin.    

\begin{cor}
\label{cor:optimal.vs}
The admissible hyper-cube predictive random set $\S$, given by 
\begin{equation}
\label{eq:optimal.box}
\S = \{u \in \RR^p: \|u\|_\infty \leq \|U\|_\infty\}, \quad U \sim \prob_U = \tdist_p(0, L; n-p-1), 
\end{equation}
is balanced in the sense of Definition~\ref{def:prs.balance} and maximin optimal for $\{A_j: j=1,\ldots,p\}$ in the sense of Definition~\ref{def:prs.optimal}.  
\end{cor}

\subsection{IM output and its interpretation}
\label{SS:interpret}

Now that we have specified the optimal predictive random set for the P-step, the C-step proceeds just as before.  It is helpful, however, to describe the IM output and its interpretation as it pertains to model uncertainty quantification.  To start, we need a bit more notation.  Let $\J \subseteq \{1,2,\ldots,p\}$ denote the collection of indices corresponding to the truly non-zero coefficients, i.e., $\theta_j \neq 0$ for all $j \in \J$.  Also, for a generic subset $J \subseteq \{1,2,\ldots,p\}$, let $\theta_J$ denote the corresponding sub-vector of $\theta$.  

Consider the following assertion about the model:
\[ C_J = \{\J \subseteq J\} = \{\theta: \theta_{J^c} = 0\}. \]
Note that $C_J^c = \bigcup_{j \in J^c} A_j$, so this new assertion is generated by the model assertions $A_1,\ldots,A_p$ from before.  Using the optimal predictive random set $\S$ in \eqref{eq:optimal.box}, we have 
\begin{equation}
\label{eq:plausibility}
\pl_Y(C_J; \S) = \prob_U\{\|U\|_\infty >  \|T_{J^c}\|_\infty\} = 1 - F(\|T_{J^c}\|_\infty), 
\end{equation}
where $F$ is the distribution function of $\|U\|_\infty$ when $U \sim \tdist_p(0, L; n-p-1)$.  The distribution function $F$ can be evaluated via Monte Carlo or the method of \citet{genz.bretz.2002}, implemented in the R package {\tt pmvt}.  

\begin{toy}
For the illustrative example in Section~\ref{SS:valid.example}, where $p=2$, the error variance $\sigma^2=1$ was known, and the design matrix was identity, the distribution function $F$ has a simpler form, namely, $F(z) = \{1 - 2 \Phi(-z)\}^2$, $z \geq 0$.  Then the plausibility functions computed in that example, and summarized in Figure~\ref{fig:simple}, are based on the expression \eqref{eq:plausibility} with this distribution function and $T$ as the pair of observations $(X_1,X_2)$.  
\end{toy} 

How to interpret the plausibility \eqref{eq:plausibility} above?  On one hand, it is clear both from intuition and from the formula that $\pl_Y(C_J; \S)$ is non-decreasing in $J$ and that the full model has plausibility 1.  On the other hand, since the optimal hyper-cube predictive random set has positive volume, the belief in any proper sub-model is zero.  So, to borrow the terminology from \citet{dempster2008}, larger models have more ``don't know'' probability, which is consistent with idea that larger models have more parameters to estimate and are more difficult to interpret.  Consequently, only the relatively small plausibilities, which correspond to relatively small models, can be interpreted.  To help understand what ``relatively small'' means, we have the following calibration result.  

\begin{thm}
\label{thm:pl.vs}
If $C_J$ is true, then $\pl_Y(C_J; \S)$ is stochastically no smaller than $\unif(0,1)$.  
\end{thm}

\begin{proof}
It is clear that, $C_J$ is true, i.e., if $\theta_{J^c}=0$, then $\|U\|_\infty$ is stochastically no smaller than $\|T_{J^c}\|_\infty$.  Then the result follows from the fact that $F(\|U\|_\infty) \sim \unif(0,1)$.   
\end{proof}

The result in Theorem~\ref{thm:pl.vs} explains why the IM plausibility's distribution function falls on or below the diagonal line in Panels~(a) and (c) in Figure~\ref{fig:simple}.  Moreover, optimality of the predictive random set suggests that the IM should also be efficient, i.e., if $C_J$ is false, then $\pl_Y(C_J; \S)$ should be stochastically smaller than $\unif(0,1)$, which is the conclusion in Panels~(b) and (d) of the same figure.  

For further illustration, consider the prostate cancer data analyzed by \citet{tibshirani1996} and others.  This study examined the association between the prostate specific antigen (PSA) level and some clinical measures among men who were about to receive a radical prostatectomy.  There were $n=97$ men with $p=8$ predictors, four of which, including cancer volume (\texttt{lcavol}), prostate weight (\texttt{lweight}), capsular penetration (\texttt{lcp}), and benign prostatic hyperplasia amount (\texttt{lbph}), were log transformed. The other four predictors were age, seminal vestical invasion (\texttt{svi}), Gleason score (\texttt{gleason}), and percentage Gleason scores 4 or 5 (\texttt{pgg45}). The response variable is the log PSA level.

We can compute the plausibilities \eqref{eq:plausibility} for each $J$, and the results are summarized in Figure~\ref{fig:prostate}, where the plausibilities are arranged by the corresponding model size.  That is, for each $s=0,1,\ldots,8$, there are $\binom{8}{s}$ plausibility values displayed vertically.  As mentioned previously, only the relatively small plausibilities in this figure can be interpreted and, since we are interested in models which are both small and sufficiently plausible, we focus on maximum plausibility for each $s$.  The maximum plausibilities for $s=1$ and $s=2$ are both rather small; the maximum value (0.09) in the $s=2$ column corresponds to the model that contains variables \texttt{lcavol} and \texttt{svi} only.  If we move to $s=3$ variables, then the largest plausibility value jumps to 0.44, corresponding to the model that adds the variable \texttt{lweight} to the previous two-variable model.  In fact, this latter three-variable model is the one that is selected by lasso.  The plot reveals that there is a lot more uncertainty, i.e., ``don't know,'' for model assertions that allow for the possibility that \texttt{lweight} is included in the true model.  We conclude that data only provides evidence to support the claim that the true model is a subset of $\{\text{\tt lcavol}, \text{\tt svi}\}$; see Section~\ref{SS:vs.procedure}.

\begin{figure}
\begin{center}
\scalebox{0.65}{\includegraphics{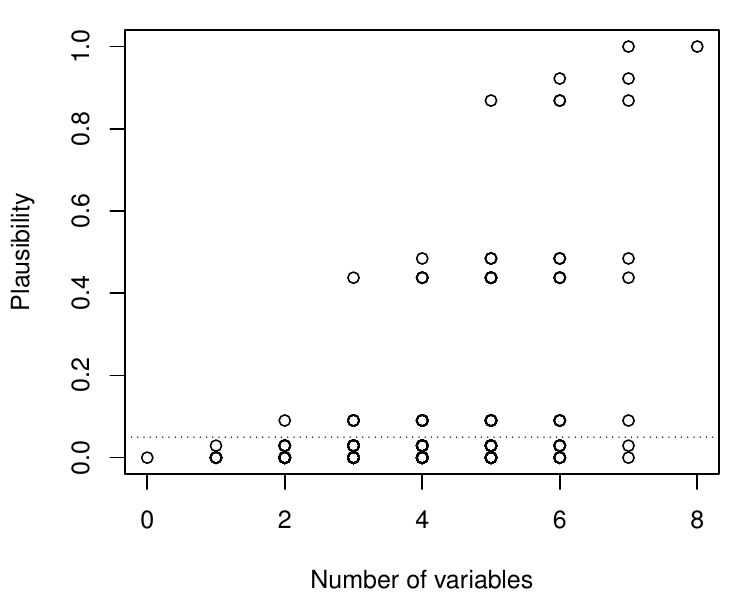}}
\end{center}
\caption{Plot of the model plausibility values, arranged by model size, for the prostate cancer data example studied in Section~\ref{SS:interpret}.}
\label{fig:prostate}
\end{figure}

\subsection{On post-selection inference}
\label{SS:post}

Using the balanced hyper-cube predictive random set $\S$ in Corollary~\ref{cor:optimal.vs}, we can construct a plausibility function for singletons $\{\theta\}$:
\[ \pl_Y(\theta; \S) := \pl_Y(\{\theta\}; \S) = 1 - F\bigl( \|\hat\sigma^{-1}(\hat\theta - \theta)\|_\infty \bigr). \]
This plausibility function satisfies $\pl_Y(\hat\theta; \S) = 1$ and is decreasing in $\theta$ away from $\hat\theta$ with hyper-cube shaped contours.  Consequently, by thresholding the plausibility function at level $\alpha \in (0,1)$, we obtain the plausibility region
\[ \{\theta: \pl_Y(\theta; \S) \geq \alpha\} = \bigl\{ \theta: F\bigl( \|\hat\sigma^{-1}(\hat\theta - \theta)\|_\infty \bigr) \leq 1-\alpha \bigr\}. \]
Since the predictive random set $\S$ is admissible, it follows from the general IM theory that the plausibility region above has nominal frequentist coverage $1-\alpha$.  Moreover, the shape of these plausibility regions is a hyper-cube, with side lengths characterized by quantiles of the $\ell_\infty$-norm of multivariate Student-t random vectors.    

An important consequence of our focus on validity simultaneously across a collection of assertions is that a naive projection of these plausibility regions to any sub-model $J \subseteq \{1,2,\ldots,p\}$ gives a new hyper-cube plausibility region that also has nominal frequentist coverage.  Since this conclusion holds uniformly in $J$, it also holds if $J$ is chosen based on data.  Such considerations are relevant in the context of post-selection inference.  As \citet{bbbzz2013} argue, when data is used to select a model, then one cannot use the model-specific distribution theory for valid inference.  This leads to the fundamental question of how to achieve valid inference when the model is chosen based on data.  They propose a procedure they call ``POSI'' for valid post-selection inference, and it turns out that their procedure is identical to that obtained by projection of the above plausibility region to a sub-model selected by data.  The take-away message here is that, by insisting on simultaneous validity over all relevant assertions, the IM approach can automatically handle the challenging post-selection inference problem.  This is a consequence of our focus on the main goal of valid uncertainty quantification about the model.

\subsection{IM-driven variable selection}
\label{SS:vs.procedure}

Section~\ref{SS:vs.assertions} showed how to optimally quantify model uncertainty, and the previous section explained how this implies valid post-selection inference.  It is also possible to develop an IM-driven variable selection procedure having desirable properties.  We want to be clear that valid uncertainty quantification about the model is our primary goal, and the result of a variable selection procedure is only one kind of summary of the IM output.   

Intuitively, a model or collection of variables $J$ is supported by data $Y$ if $\pl_Y(C_J; \S)$ is not too small.  In other words, good models are those which are ``sufficiently plausible'' given data.  Following this idea, and the remarks in Section~\ref{SS:interpret}, an IM-driven approach for variable selection would be to fix $\alpha \in (0,1)$ and then pick the smallest collection $J$ such that $\pl_Y(C_J; \S) > \alpha$.  That is, define 
\begin{equation}
\label{eq:vs.rule}
\hat J_\alpha(Y) = \text{smallest set $J$ such that $\pl_Y(C_J; \S) > \alpha$}. 
\end{equation}
We claim that the IM-driven procedure \eqref{eq:vs.rule} satisfies a \emph{selection validity} property:
\begin{equation}
\label{eq:selection.validity}
\prob_{Y|\J}\{\hat J_\alpha(Y) \subseteq \J\} \geq 1-\alpha, \quad \forall \; \J \subseteq \{1,2,\ldots,p\}.
\end{equation}

\begin{thm}
\label{thm:im.vs}
The procedure \eqref{eq:vs.rule} has the selection validity property \eqref{eq:selection.validity}.
\end{thm}

\begin{proof}
$\pl_Y(C_\J; \S) > \alpha$ implies $\hat J_\alpha(Y) \subseteq \J$, so apply Theorem~\ref{thm:pl.vs}. 
\end{proof}


To implement the above procedure, it is not necessary to evaluate the plausibility function at $C_J$ for each $J \subseteq \{1,2,\ldots,p\}$.  In fact, the plausibility depends on the value of $\|T_{J^c}\|_\infty$, so we only need to look at $p$ different models, based on a sorting of the t-statistics by their magnitude.  Let $\pi$ be a permutation that ranks the $T$ values in descending order according to their magnitudes, i.e., $|T_{\pi(1)}| > |T_{\pi(2)}| > \cdots > |T_{\pi(p)}|$, and then compute 
\begin{equation}
\label{eq:vs.pl}
\mpl_Y(s) = \max_{J:|J|=s} \pl_Y(C_J; \S) = 1 - F(|T_{\pi(s+1)}|), \quad s=0,1,\ldots,p-1, 
\end{equation}
and $\mpl_Y(p) \equiv 1$.  This is a naive ``marginal'' plausibility function for the \emph{size} of the model \citep[][Sec.~G]{shafer1987}.  If $s_\alpha^\star$ is the smallest $s$ such that $\mpl_Y(s) > \alpha$, then 
\begin{equation}
\label{eq:im.vs.proc}
\hat J_\alpha(Y) = \pi^{-1}(\{1,\ldots,s_\alpha^\star\}).
\end{equation}
For example, for the prostate cancer data discussed in Section~\ref{SS:interpret}, based on the summary in Figure~\ref{fig:prostate}, if set the cutoff at $\alpha=0.05$, then the IM-based rule \eqref{eq:vs.rule} selects the two-variable model that includes \texttt{lcavol} and \texttt{svi}.  

A simulation study is performed to assess the performance of this variable selection procedure.  We generate 1000 data sets consisting of observations $Y \sim \nm_n(X\beta, \sigma^2 I_n)$, where the rows of the predictor variable matrix are draws from $\nm_p(0,\Omega)$ with an autoregressive correlation structure, i.e., $\Omega_{jk} = \rho^{|j-k|}$ for all $j,k$.  This model was used by \citet{tibshirani1996} in his simulation study.  We take $\sigma=3$ and $\rho=0.5$ and consider two different scenarios:
\begin{description}
\item[\it Scenario~1.] $p=7$, $\beta=(3, 1.5, 0, 0, 2, 0, 0)^\top$;
\vspace{-2mm}
\item[\it Scenario~2.] $p=20$, $\beta=(0.85 \, 1_{10}^\top, 0_{10}^\top)^\top$.
\end{description}
The relative comparisons based on simulations under some different settings were the same as here so these are not presented here.  For each scenario, we vary the sample size $n$ from 50 to 5000. For each simulated data set, the IM-driven variable selection procedure is applied.  Performance of this approach is compared to the lasso, where the tuning parameter is chosen using the 10-fold cross validation, the adaptive lasso, AIC, and BIC.  AIC and BIC variable selection is performed using exhaustive search of all possible models.  Results of these approaches are summarized in Figures~\ref{fig:scenario1} and \ref{fig:scenario2} for the two scenarios.  Specifically, we plot the percentage of selecting the true model, where all important variables are selected (Panel A), a subset of the true model, where some but not all important variables and no unimportant variables are selected (Panel B), a superset of the true model, where all important variables and at least one unimportant variable are selected (Panel D), and the other models, where at least one important variable is not selected and at least one unimportant variable is selected (Panel E). In order to show the selection validity of the IM-driven approach, we also plot the percentage of selecting the true model or a subset of the true model, where some or all important variables are selected while no unimportant variables are selected (Panel C). 

In both scenarios, the IM-driven approach performs well in terms of how frequently it selects the true model---similar to BIC but better than lasso and AIC.  When it misses the true model, the IM-driven approach tends to select a subset of the important variables in both scenarios.  Other approaches, on the other hand, are more likely to select unimportant variables.  This is especially the case in Scenario~2 with a larger number of variables, where these approaches miss important variables and incorrectly select unimportant variables (Panel E). Overall, the IM-based approach maintains the 95\% level for selecting the important variables without picking up unimportant variables (Panel C), a consequence of Theorem~\ref{thm:im.vs}. The other approaches, except for BIC and adaptive lasso, are generally far off the mark.  That BIC and adaptive lasso can reach above 95\% for selecting true model or subset of true model is due to their asymptotic variable selection consistency property, not because they are properly calibrated at any fixed $n$.  

The take-away message is that the variable selection procedure \eqref{eq:vs.rule} derived from our IM is surely competitive with some of the standard methods used in the literature.  Other ad hoc selection procedures can be used which may beat our IM-driven variable selection procedure with respect to some criterion, but these procedures will always fall short of providing uncertainty quantification about the model.

\begin{figure}
\begin{center}
\scalebox{0.5}{\includegraphics{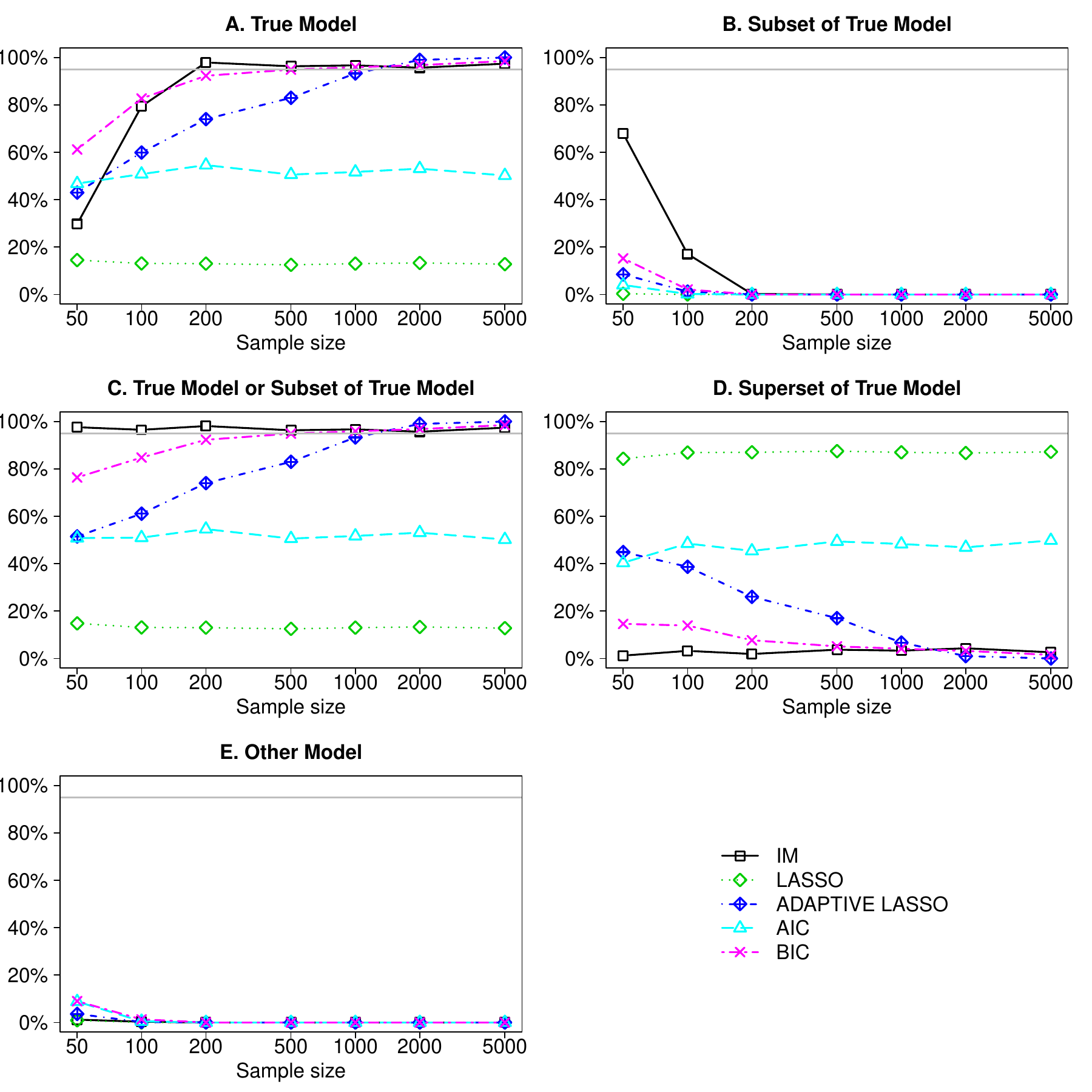}}
\end{center}
\caption{Percentage of selecting the true model (panel A), subset of the true model where some but not all important variables is selected (panel B), true model or subset of true model where some or all important variables are selected (panel C), superset if true model where all important variables and at least one unimportant variable are selected (panel D), and other model where at least one important variable is not selected and at least one unimportant variable is selected (panel E) for Scenario~1.}
\label{fig:scenario1}
\end{figure}

\begin{figure}
\begin{center}
\scalebox{0.5}{\includegraphics{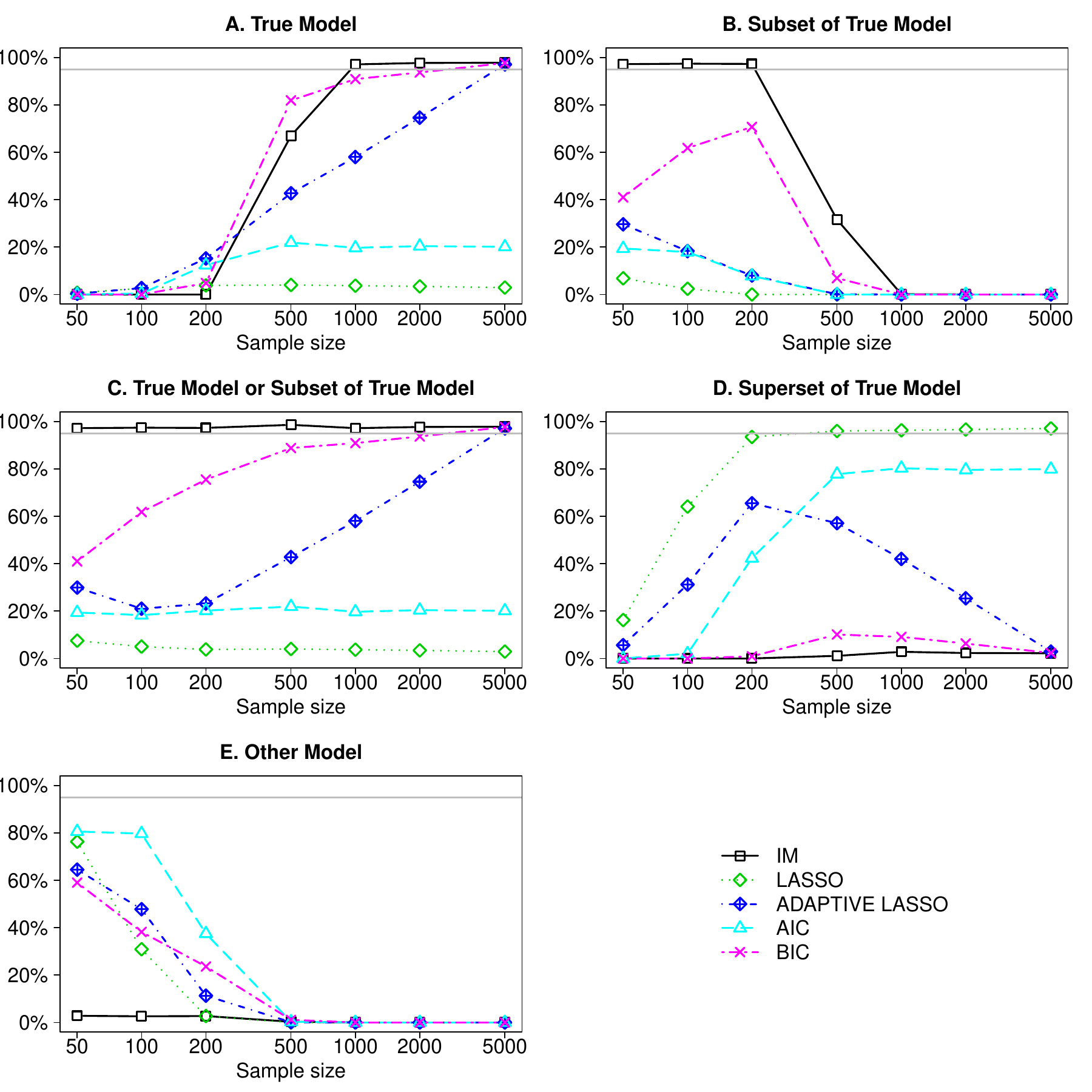}}
\end{center}
\caption{Same explanation as for Figure~\ref{fig:scenario1} but for Scenario~2.}
\label{fig:scenario2}
\end{figure}

\section{Discussion}
\label{S:discuss}

This paper introduces the concept of multiple simultaneous assertions in the IM framework, and develops a theory of optimal predictive random sets.  These general principles are then applied, in a regression setting, to provide a valid quantification of the uncertainty in the model.  To our knowledge, the IM approach is the only method known to satisfy this important property.  The IM's simultaneous validity property sheds light on the post-selection inference problem, and also leads to a variable selection procedure with desirable frequentist error rate control.  Our simulation study demonstrates that, overall, the proposed IM-based procedure performs as good or better than several standard methods, which suggests that the IM's emphasis on valid probabilistic uncertainty quantification does not come at the cost of decreased efficiency.  Extending the developments here for application in generalized linear model settings is a focus of ongoing research.  

The IM approach has already been applied to other problems that involve multiplicity, such as large-scale multinomial inference in genome wide association studies and multiple testing problems \citep{liu.xie.genome, liu.xie.imtest}.  We expect that the optimality considerations here, applied to those problems, would lead to some overall improvements.  Perhaps the most important question is how to extend the developments in this paper to the high-dimensional regression context.  
The basic principles of optimal predictive random sets developed here for simultaneous complex assertions are not specific to the $p < n$ case; however, the initial dimension reduction steps do not carry over directly to the $p \gg n$ case.  We expect that once the A-step can be completed for $p \gg n$, ideas similar to those presented here can be applied.  

\section*{Acknowledgements}

The authors thank the Editor, Associate Editor, and referees for their valuable comments on an earlier draft of this manuscript.  This work was partially supported by the U.S.~National Science Foundation: DMS--1007678, DMS--1208833, and DMS--1208841.  The authors also thank Professor Jan Hannig for suggesting a possible connection between our work and that of \citet{bbbzz2013}.

\appendix

\section{Proofs}
\label{S:proofs}

\subsection{Proof of Theorem~\ref{thm:optimal.simple}}
\label{proof:optimal.simple}

Proposition~1 in \citet{imbasics} shows that $\bel_x(A;\S) \leq \prob_U\{\UU_A(x)\}$, for any admissible $\S$ and any $x$.  By assumption, the collection $\{\UU_A(x): x \in \XX\}$ is nested, so it can be taken as the support of an admissible predictive random set.  In this case, since $\prob_\S$ satisfies \eqref{eq:natural}, we have that $\bel_x(A;\S) = \prob_U\{\UU_A(x)\}$ for all $x$.  Therefore, the belief function attains its upper bound for each $x$, hence, it is optimal.

\subsection{Proof of Theorem~\ref{thm:intersection}}
\label{proof:intersection}

Since $A_1$ and $A_2$ are simple assertions, the respective optimal predictive random sets $\S_1$ and $\S_2$ have focal elements given by $\UU_{A_1}(x)$ and $\UU_{A_2}(x)$, as $x$ ranges over $\XX$.  Without loss of generality, assume that the predictive random set $\T$ for $A=A_1 \cup A_2$ is admissible.  That is, assume $\T$ has a nested support $\TT$ and that $\prob_\T$ satisfies \eqref{eq:natural}.  Define 
\[ S_j(T) = \text{closure}\Bigl\{ \bigcap_{y: \UU_{A_j}(y) \supset T} \UU_{A_j}(y) \Bigr\} \quad j=1,2, \quad T \in \TT.  \]
Collect intersections of these sets, $\SS=\{S_1(T) \cap S_2(T): T \in \TT\}$.  Since $\TT$ is nested and the function $T \mapsto S_1(T) \cap S_2(T)$ is monotone, the collection $\SS$ is also nested.  Define a new predictive random set $\S$, supported on $\SS$, with the natural measure $\prob_\S$ as in \eqref{eq:natural}.  This predictive random set satisfies the conditions stated in the theorem, i.e., admissible and has focal elements as intersections of the respective optimal predictive random set focal elements.  We need to show that $T \subset \UU_A(x)$ if and only if $S(T) \subset \UU_A(x)$, where $S(T)=S_1(T) \cap S_2(T)$.  One direction is easy, since it is clear that $T \subset S(T)$.  For the other direction, we want to show that $T \subset \UU_A(x)$ implies $S(T) \subset \UU_A(x)$.  Since $\UU_{A_1}(x)$ and $\UU_{A_2}(x)$ are disjoint, and the union is $\UU_A(x)$, $T \subset \UU_A(x)$ implies that either $T \subset \UU_{A_1}(x)$ or $T \subset \UU_{A_2}(x)$.  By definition of $S_1(T)$ and $S_2(T)$, it follows that 
\[ S_1(T) \subset \UU_{A_1}(x) \quad \text{or} \quad S_2(T) \subset \UU_{A_2}(x). \]
In either case, $S(T) = S_1(T) \cap S_2(T)$ is contained in $\UU_A(x) = \UU_{A_1}(x) \cup \UU_{A_2}(x)$, which completes the argument that $\{T \subset \UU_A(x)\}$ and $\{S(T) \subset \UU_A(x)\}$ are equivalent.  Finally, since $T \subset S(T)$, we get $\prob_U(T) \leq \prob_U(S(T))$ for all $T \in \TT$ and, therefore, 
\[ \prob_\T\{\T \subset \UU_A(x)\} = \sup_{T: T \subset \UU_A(x)} \prob_U(T) \leq \sup_{T: S(T) \subset \UU_A(x)} \prob_U(S(T)) = \prob_\S\{\S \subset \UU_A(x)\}. \]
The left-hand side is $\bel_x(A;\T)$ and the right-hand side is $\bel_x(A;\S)$, and the inequality holds for all $x$, completing the proof.

\subsection{Proof of Theorem~\ref{thm:new.intersection}}
\label{proof:new.intersection}

The proof here is similar to that of Theorem~\ref{thm:intersection}.  Consider collection $\{A_j: j \in J\}$ of complex assertions, where each $A_j$ can be written as a union of disjoint simple assertions, i.e., $A_j = A_{j1} \cup A_{j2}$, where $A_{j1} \cap A_{j2} = \varnothing$ and the a-events $\UU_{A_{j1}}(\cdot)$ and $\UU_{A_{j2}}(\cdot)$ are nested.  By Theorem~\ref{thm:intersection}, we know that the optimal predictive random sets for $A_j$, $j \in J$, have focal elements $\SS_j = \{S_j(v): v \in V\}$, indexed by a set $V$, which are intersections of a-events.  That is, $S_j(v)$ is (the closure of) $\UU_{A_{j1}}(x_{j1}(v)) \cap \UU_{A_{j2}}(x_{j2}(v))$ for some $x_{j1}(v)$ and $x_{j2}(v)$.  

Without loss of generality, assume that the candidate predictive random set $\T$ for the assertion $A$ generated by $\{A_j: j \in J\}$ is admissible in the sense of Definition~\ref{def:admissible}.  Given a focal element $T \in \TT$ of $\T$, define 
\[ \tilde S_j(T) = \bigcap_{v: S_j(v) \supset T} S_j(v), \quad j \in J. \]
Next, set $\tilde S(T) = \bigcap_{j \in J} \tilde S_j(T)$ and define $\SS = \{\tilde S(T): T \in \TT\}$.  Now take $\S$ to have support $\SS$ and natural measure $\prob_\S$ as in \eqref{eq:natural}; this $\S$ satisfies the conditions of the theorem.  It remains to show that $\S$ is more efficient than $\T$.

As in the proof of Theorem~\ref{thm:intersection}, we need to show that $T \subset \UU_A(x)$ if and only if $\tilde S(T) \subset \UU_A(x)$ for each $x$ and $T \in \TT$.  By construction, $T \subset \tilde S(T)$, so one direction is handled.  For the other direction, assume that $T \subset \UU_A(x)$.  Then $T \subset \UU_{A_j}(x)$ for some $j \in J$ and, since $A_j$ splits as a disjoint union of simple assertions, we get that $T \subset \UU_{A_{j1}}(x)$ or $T \subset \UU_{A_{j2}}(x)$.  Then $\tilde S_j(T) \subset \UU_{A_{j1}} \cap \UU_{A_{j2}}(x)$ and, consequently, the no-bigger $\tilde S(T)$ must be a subset of $\UU_A(x)$.  Since $T \subset \UU_A(x)$ if and only if $\tilde S(T) \subset \UU_A(x)$, the claimed superiority of $\S$ to $\T$ follows just like in the last part of the proof of Theorem~\ref{thm:intersection}.

\subsection{Proof of Lemma~\ref{lem:belief}}
\label{proof:belief}

Recall that $\bel_x(A;\S)$ is defined as $\prob_\S\{\S \subset \UU_A(x)\}$.  Since $\S$ has the natural measure \eqref{eq:natural}, we can write, for any $b \in [0,1]$, 
\begin{align*}
\bel_x(A;\S) > b & \iff \prob_\S\{\S \subset \UU_A(x)\} > b \\
& \iff \sup_{r: S_r \subset \UU_A(x)} \prob_U(S_r) > b \\
& \iff \sup_{r: S_r \subset \UU_A(x)} (1-r) > b \\
& \iff \UU_A(x) \supset S_{1-b}.
\end{align*}
Therefore, we have that 
\[ \prob_{X|\theta}\{\bel_X(A;\S) > b\} = \prob_{X|\theta}\{\UU_A(X) \supset S_{1-b}\}, \quad \forall \; b \in [0,1], \]
which proves the claim, with $r=1-b$.

\subsection{Proof of Theorem~\ref{thm:balance.optimal}}
\label{proof:balance.optimal}

Take any predictive random set $\S$ as in Theorem~\ref{thm:new.intersection}, and let $S$ be a generic focal element.   Take any assertion $A$ generated by $\{A_j: j \in J\}$.  Define the core of $S$ as $S^\circ = \bigcap_{g \in \G_A} g S$; note that $S^\circ$ is balanced and satisfies $S^\circ \subset g S$ for all $g \in \G_A$.   Then, 
\[ \prob_{X|\theta}\{\UU_A(X) \supset S^\circ\} \geq \prob_{X|\theta}\{\UU_A(X) \supset g S\}, \quad \forall \; g \in \G_A, \quad \forall \; \theta \in A, \]
with strict inequality in general.  Maximin optimality requires that we choose the focal elements to maximize the minimum (over $g$) of the right-hand side of the above display.  However, we can clearly attain the upper bound above by taking the focal element $S$ equal to its core, i.e., balanced.  Therefore, balance implies maximin optimality.

\bibliographystyle{apalike}
\bibliography{/Users/rgmartin/Dropbox/Research/mybib}

\begin{thebibliography}{}

\bibitem[Akaike, 1973]{akaike1973}
Akaike, H. (1973).
\newblock Information theory and an extension of the maximum likelihood
  principle.
\newblock In {\em Second {I}nternational {S}ymposium on {I}nformation {T}heory
  ({T}sahkadsor, 1971)}, pages 267--281. Akad\'emiai Kiad\'o, Budapest.

\bibitem[Berger et~al., 2009]{bergerbernardosun2009}
Berger, J.~O., Bernardo, J.~M., and Sun, D. (2009).
\newblock The formal definition of reference priors.
\newblock {\em Ann. Statist.}, 37(2):905--938.

\bibitem[Berk et~al., 2013]{bbbzz2013}
Berk, R., Brown, L., Buja, A., Zhang, K., and Zhao, L. (2013).
\newblock Valid post-selection inference.
\newblock {\em Ann. Statist.}, 41(2):802--837.

\bibitem[B{\"u}hlmann, 2013]{buhlmann2013}
B{\"u}hlmann, P. (2013).
\newblock Statistical significance in high-dimensional linear models.
\newblock {\em Bernoulli}, 19(4):1212--1242.

\bibitem[Clyde and George, 2004]{clydegeorge2004}
Clyde, M. and George, E.~I. (2004).
\newblock Model uncertainty.
\newblock {\em Statist. Sci.}, 19(1):81--94.

\bibitem[Dempster, 2008]{dempster2008}
Dempster, A.~P. (2008).
\newblock The {D}empster--{S}hafer calculus for statisticians.
\newblock {\em Internat. J. Approx. Reason.}, 48(2):365--377.

\bibitem[Ermini~Leaf and Liu, 2012]{leafliu2012}
Ermini~Leaf, D. and Liu, C. (2012).
\newblock Inference about constrained parameters using the elastic belief
  method.
\newblock {\em Internat. J. Approx. Reason.}, 53(5):709--727.

\bibitem[Fraser, 2011]{fraser2011}
Fraser, D. A.~S. (2011).
\newblock Is {B}ayes posterior just quick and dirty confidence?
\newblock {\em Statist. Sci.}, 26(3):299--316.

\bibitem[Fraser et~al., 2010]{fraser.reid.marras.yi.2010}
Fraser, D. A.~S., Reid, N., Marras, E., and Yi, G.~Y. (2010).
\newblock Default priors for {B}ayesian and frequentist inference.
\newblock {\em J. R. Stat. Soc. Ser. B Stat. Methodol.}, 72(5):631--654.

\bibitem[Genz and Bretz, 2002]{genz.bretz.2002}
Genz, A. and Bretz, F. (2002).
\newblock Comparison of methods for the computation of multivariate {$t$}
  probabilities.
\newblock {\em J. Comput. Graph. Statist.}, 11(4):950--971.

\bibitem[Hannig et~al., 2015]{hannig.review.2015}
Hannig, J., Iyer, H., Lai, R. C.~S., and Lee, T. C.~M. (2015).
\newblock Generalized fiducial inference: A review.
\newblock {\it J. Amer. Statist. Assoc.}, to appear.

\bibitem[Lai et~al., 2015]{lai.hannig.lee.2013}
Lai, R. C.~S., Hannig, J., and Lee, T. C.~M. (2015).
\newblock Generalized fiducial inference for ultrahigh dimensional regression.
\newblock {\em J. Amer. Statist. Assoc.}, 110:760--772.

\bibitem[Liu and Xie, 2014a]{liu.xie.genome}
Liu, C. and Xie, J. (2014a).
\newblock Large scale two sample multinomial inferences and its applications in
  genome-wide association studies.
\newblock {\em Internat. J. Approx. Reason.}, 55(1, part 3):330--340.

\bibitem[Liu and Xie, 2014b]{liu.xie.imtest}
Liu, C. and Xie, J. (2014b).
\newblock Probabilistic inference for multiple testing.
\newblock {\em Internat. J. Approx. Reason.}, 55(2):654--665.

\bibitem[Lockhart et~al., 2014]{lttt2014}
Lockhart, R., Taylor, J., Tibshirani, R.~J., and Tibshirani, R. (2014).
\newblock A significance test for the lasso.
\newblock {\em Ann. Statist.}, 42(2):413--468.

\bibitem[Martin and Liu, 2013]{imbasics}
Martin, R. and Liu, C. (2013).
\newblock Inferential models: A framework for prior-free posterior
  probabilistic inference.
\newblock {\em J. Amer. Statist. Assoc.}, 108(501):301--313.

\bibitem[Martin and Liu, 2014]{impval}
Martin, R. and Liu, C. (2014).
\newblock A note on $p$-values interpreted as plausibilities.
\newblock {\em Statist. Sinica}, 24:1703--1716.

\bibitem[Martin and Liu, 2015a]{imcond}
Martin, R. and Liu, C. (2015a).
\newblock Conditional inferential models: Combining information for prior-free
  probabilistic inference.
\newblock {\em J. R. Stat. Soc. Ser. B}, 77(1):195--217.

\bibitem[Martin and Liu, 2015b]{immarg}
Martin, R. and Liu, C. (2015b).
\newblock Marginal inferential models: Prior-free probabilistic inference on
  interest parameters.
\newblock {\em J. Amer. Statist. Assoc.}, 110:1621--1631.

\bibitem[Martin et~al., 2015]{martin.mess.walker.eb}
Martin, R., Mess, R., and Walker, S.~G. (2015).
\newblock Empirical {B}ayes posterior concentration in sparse high-dimensional
  linear models.
\newblock {\it Bernoulli}, to appear; {\tt arXiv:1406.7718}.

\bibitem[Molchanov, 2005]{molchanov2005}
Molchanov, I. (2005).
\newblock {\em Theory of Random Sets}.
\newblock Probability and Its Applications (New York). Springer-Verlag London
  Ltd., London.

\bibitem[O'Hara and Sillanp{\"a}{\"a}, 2009]{ohara2009}
O'Hara, R.~B. and Sillanp{\"a}{\"a}, M.~J. (2009).
\newblock A review of {B}ayesian variable selection methods: what, how and
  which.
\newblock {\em Bayesian Anal.}, 4(1):85--117.

\bibitem[Schwarz, 1978]{schwarz1978}
Schwarz, G. (1978).
\newblock Estimating the dimension of a model.
\newblock {\em Ann. Statist.}, 6(2):461--464.

\bibitem[Shafer, 1987]{shafer1987}
Shafer, G. (1987).
\newblock Belief functions and possibility measures.
\newblock In Bezdek, J.~C., editor, {\em The Analysis of Fuzzy Information,
  Vol. 1: Mathematics and Logic}, pages 51--84. CRC.

\bibitem[Tibshirani, 1996]{tibshirani1996}
Tibshirani, R. (1996).
\newblock Regression shrinkage and selection via the lasso.
\newblock {\em J. Roy. Statist. Soc. Ser. B}, 58(1):267--288.

\bibitem[Tibshirani, 2011]{tibshirani2011}
Tibshirani, R. (2011).
\newblock Regression shrinkage and selection via the lasso: a retrospective.
\newblock {\em J. R. Stat. Soc. Ser. B Stat. Methodol.}, 73(3):273--282.

\bibitem[Xie and Singh, 2013]{xie.singh.2012}
Xie, M. and Singh, K. (2013).
\newblock Confidence distribution, the frequentist distribution of a parameter
  -- a review.
\newblock {\em Int. Statist. Rev.}, 81(1):3--39.

\bibitem[Zellner, 1986]{zellner1986}
Zellner, A. (1986).
\newblock On assessing prior distributions and {B}ayesian regression analysis
  with {$g$}-prior distributions.
\newblock In {\em Bayesian Inference and Decision Techniques}, volume~6 of {\em
  Stud. Bayesian Econometrics Statist.}, pages 233--243. North-Holland,
  Amsterdam.

\bibitem[Zhang et~al., 2011]{imreg}
Zhang, Z., Xu, H., Martin, R., and Liu, C. (2011).
\newblock Inferential models for linear regression.
\newblock {\em Pak. J. Statist. Oper. Res.}, 7:413--432.

\bibitem[Zou, 2006]{zou2006}
Zou, H. (2006).
\newblock The adaptive lasso and its oracle properties.
\newblock {\em J. Amer. Statist. Assoc.}, 101(476):1418--1429.

\bibitem[Zou and Hastie, 2005]{zou.hastie.2005}
Zou, H. and Hastie, T. (2005).
\newblock Regularization and variable selection via the elastic net.
\newblock {\em J. R. Stat. Soc. Ser. B Stat. Methodol.}, 67(2):301--320.

\end{thebibliography}

\pagebreak

\begin{center}
{\LARGE Supplementary Material}
\end{center}

\section*{S \quad Measure-theoretic details}

\subsection*{S.1 \quad Random sets and the admissibility condition}

\citet{molchanov2005} gives a comprehensive treatment of the theory of random sets.  Our goal here is present the minimal amount of technical details necessary to understand our analysis involving predictive random sets.  To start, the standard theory of random sets focuses on the case of closed random sets, i.e., random sets whose values are closed sets (with probability~1).  Write $\SS$ for the support of our random set $\S$, a collection of subsets of the space $\UU$; we assume that $\SS$ contains both $\varnothing$ and $\UU$.  Let $\UU$ be a separable metric space and take $\U$ to be a $\sigma$-algebra of subsets that contains all the closed subsets of $\UU$.  Assume that each set, or focal element, $S \in \SS$ is closed, relative to the topology on $\UU$.   

To define the random set $\S$, consider a probability space $(\Omega, \A, \prob)$, and let $\S$ be mapping from $\Omega$ to $\SS$ which is measurable in the sense that 
\[ \{\omega: \S(\omega) \cap K \neq \varnothing\} \in \A, \quad \text{for all compact $K \subseteq \UU$}. \]
We define the distribution $\prob_\S$ of $\S$ as the push-forward measure $\prob \S^{-1}$.  Also, for compact $K$, the event $\{\S \subset K\}$ is measurable, so the probability $\prob_\S\{\S \subset K\}$ in \eqref{eq:natural} and also below makes sense.  Moreover, since $\UU$ is separable and $\S$ is closed, the indicator stochastic process $\{I_\S(u): u \in \UU\}$ is separable and, by Proposition~4.10 in \citet{molchanov2005}, the distribution $\prob_\S$ is determined by probabilities assigned to the events $\{\S \subset K\}$, and can be extended to include arbitrary sets $K$.

An important class of examples are the predictive random sets described in Corollary~1 of \citet{imbasics}; the ``default'' predictive random set that has been frequently used in the IM literature, as well as the optimal predictive random set employed in this paper, are members of this class.  In particular, start by taking the probability space $(\Omega, \A, \prob)$ to be $(\UU, \U, \prob_U)$.  Next, for a continuous function $h: \UU \to \RR$, define the set-valued mapping $\psi_h$ on $\UU$ as 
\[ \psi_h(u) = \{u' \in \UU: h(u') \leq h(u)\}, \quad u \in \UU. \]
Then the focal elements $S \in \SS$ are closed level sets of the function $h$, and if $\U$ is rich enough to contain all the closed sets, then $\psi_h$ is a measurable function and, consequently, $\S = \psi_h(U)$, for $U \sim \prob_U$, is a closed random set.  

In the IM context, as a consequence of the results in \citet{imbasics}, we focus only on predictive random sets with nested support $\SS$, where nested means that, for any two focal elements, one is a subset of the other.  Such supports can, and usually are, constructed as in the small example above, but requiring that the function $h$ acts like a metric, i.e., $h$ has a unique minimizer and that $h(u)$ is increasing as $u$ moves further away from that minimizer.  The default predictive random set, for the case $\UU=[0,1]$, takes $h(u) = |u-0.5|$, which is of the form just described.  Using the terminology in \citet{shafer1987}, we could call a predictive random set with nested support \emph{consonant} and, in addition to the admissibility properties demonstrated in \citet{imbasics}, consonant random sets have the simplest distributional properties.  

In our restriction to admissible predictive random sets, as described in Definition~\ref{def:admissible}, there is a more pressing concern, namely, does there exist a predictive random set $\S$, supported on a given nested collection $\SS$ of closed $\prob_U$-measurable subsets, such that 
\[ \prob_\S\{\S \subset K\} = \sup_{S \in \S: S \subset K} \prob_U(S). \]
Specifically, while the right-hand side above is a well-defined quantity, it is not obvious that there is a random set $\S$ with distribution $\prob_\S$ that satisfies the above equality.  However, the famous Choquet capacity theorem \citep[][Theorem~1.13]{molchanov2005} can be applied to show existence of such an $\S$.  In particular, from the above display, one can identify the condition that the corresponding capacity function of $\S$ must satisfy and, from the close connection with the probability measure $\prob_U$, the required upper semi-continuity and completely alternating properties can be checked relatively easily.

\subsection*{S.2 \quad Measurability assumptions}

In addition to measurability of set-valued mappings, as discussed above, there are questions about some more familiar forms of measurability, pertaining to the a-events $\UU_A(x)$ for a given assertion $A \subseteq \Theta$.  Recall that the auxiliary variable space $\UU$ is equipped with a $\sigma$-algebra $\U$ of $\prob_U$-measurable subsets, assumed to contain all the closed sets.  Next, equip the sample space $\XX$ with a $\sigma$-algebra $\X$ and a $\sigma$-finite measure $\mu$, where $\prob_{X|\theta} \ll \mu$ for each $\theta$.  Then the two key measurability assumptions are:
\begin{itemize}
\item $\UU_A(x) \in \U$ for all relevant $A$ and for $\mu$-almost all $x$, and 
\vspace{-2mm}
\item $\{x: \UU_A(x) \supset S\} \in \X$ for all focal elements $S \in \SS$.
\end{itemize}
General sufficient conditions can be given based properties of the association mapping $x=a(\theta,u)$ and of the assertion $A$.  It is relatively easy to check the above conditions directly in a particular example and, moreover, the conditions might fail only in non-standard problems.  

\end{document}